\documentclass[12pt,twoside]{amsart}

\usepackage{amsmath, amssymb, amscd, paralist,tabularx,supertabular,amsmath, verbatim, amsthm, amssymb, mathrsfs, manfnt,times,latexsym,amscd,graphicx, color, hyperref}
\usepackage{amsmath, amsthm, amssymb}
\usepackage{paralist}  
\usepackage{manfnt}    
\usepackage{url}
\usepackage[all]{xy}

\usepackage{fullpage}

\newcommand{\C}{\mathbb C}

\newcommand{\Q}{\mathbb Q}
\newcommand{\R}{\mathbb R}

\newcommand{\frakX}{\mathfrak{X}}

\newcommand{\calG}{\mathcal{G}}

\newcommand{\calK}{\mathcal{K}}

\newcommand{\calO}{\mathcal{O}}

\numberwithin{equation}{section}

\theoremstyle{remark}

\newtheorem*{rmk*}{Remark}
\newtheorem{rmk}[equation]{Remark}

\theoremstyle{plain}

\newtheorem{theorem}[equation]{Theorem}
\newtheorem*{thma}{Main Theorem}
\newtheorem*{theoremrefom}{Theorem 2.1. (Algebraic Group Reformulation)}

\newtheorem{lemma}[equation]{Lemma}

\title{On the isospectral orbifold-manifold problem for nonpositively curved locally symmetric spaces}

\author{Benjamin Linowitz}
\address{Department of Mathematics\\ 
530 Church Street\\
Ann Arbor, MI, 48109 USA}
\email{linowitz@umich.edu}

\author{Jeffrey S. Meyer}
\address{Department of Mathematics\\ 
University of Oklahoma\\ 
Norman, OK 73019 USA}
\email[]{jmeyer@math.ou.edu}

\begin{document}

\begin{abstract} 
An old problem asks whether a Riemannian manifold can be isospectral to a Riemannian orbifold with nontrivial singular set. In this short note we show that under the assumption of Schanuel's conjecture in transcendental number theory, this is impossible whenever the orbifold and manifold in question are length-commensurable compact locally symmetric spaces of nonpositive curvature associated to simple Lie groups.
\end{abstract}

\maketitle

\section{Introduction}

This paper concerns a problem in the inverse spectral geometry of Riemannian orbifolds. These orbifolds generalize the notion of Riemannian manifolds and are defined so that each point has a neighborhood which can be identified with the quotient of an open subset of a Riemannian manifold by a finite group of isometries. We will omit the precise definition of Riemannian orbifolds (see for instance \cite{Satake, Thurston, OrbifoldStringy, GordonSurvey}) because all of the orbifolds considered in this paper belong to the well-behaved class of \emph{good orbifolds}.

A good Riemannian orbifold $\calO$ is the quotient of a Riemannian manifold $(M,g)$ by a discrete group $\Gamma$ acting by isometries on $M$. We define the space $C^\infty(\calO)$ of smooth functions on $\calO$ to be the space $C^\infty(M)^{\Gamma}$ of $\Gamma$-invariant functions on $M$. The Laplace operator $\Delta_g$ of $(M,g)$ commutes with isometries and therefore preserves $\Gamma$-invariance and hence the space $C^\infty(\calO)$. Its restriction to $C^\infty(\calO)$ is the Laplace operator (acting on functions) of $\calO$. Whenever $\calO$ is compact and connected the Laplace operator of $\calO$ has a discrete spectrum of eigenvalues $0=\lambda_0<\lambda_1\leq\cdots$ where $\lambda_k\rightarrow \infty$ as $j\rightarrow \infty$ and where each eigenvalue occurs with finite multiplicity. Two compact Riemannian orbifolds are said to be isospectral if they have the same eigenvalue spectra. For a more thorough discussion of the spectral geometry of orbifolds we refer the reader to the survey of Gordon \cite{GordonSurvey}.

It is natural to ask whether the spectrum detects the presence of singularities; that is, can a Riemannian manifold be isospectral to a Riemannian orbifold with a nontrivial singular set? While this problem remains open, several partial results have been proven. In dimension two, Dryden and Strohmaier have shown \cite{Dryden} that the Laplace spectra of two orbisurfaces are the same if and only if the length spectrum and the number of cone points of each order are the same. An application of the heat kernel expansion shows that an even (respectively odd) dimensional orbifold whose singular set contains an odd (respectively even) dimensional stratum cannot be isospectral to a manifold \cite{Gordon2008}. If an orbifold with nontrivial singular set and a manifold have a common Riemannian cover then they cannot be isospectral \cite{GorRos}. A result with a similar flavor was obtained by Sutton \cite{Sutton}, who showed that if a good orbifold with nontrivial singular set and a manifold are isospectral then they do not admit nontrivial Riemannian covers that are isospectral. In a different direction, if one instead considers the Hodge Laplacian acting on $p$-forms in the middle degree, then it is known \cite{GorRos} that an orbifold with nontrivial singular set can be isospectral to a manifold. We note that Rossetti, Schueth, and Weilandt \cite{Rossetti} have produced examples of isospectral orbifolds whose singular sets, while nontrivial, are remarkably different. One of their examples for instance, exhibits Sunada-isospectral (hence strongly isospectral) orbifolds with different maximal isotropy orders.

In this paper we will consider the setting of compact locally symmetric spaces of nonpositive curvature which are associated to simple Lie groups and show that under mild assumptions, an orbifold with nontrivial singular set cannot be isospectral to a manifold. Before stating our main theorem however, we set up some notation. Given a real semisimple Lie group $\calG$ with maximal compact subgroup $\calK$, let $\frakX=\calK\backslash \calG$ be the associated symmetric space. If $\Gamma\subset \calG$ is a discrete subgroup then the quotient $\frakX_{\Gamma}=\frakX/\Gamma$ is a locally symmetric orbifold which is a manifold when $\Gamma$ is torsion-free. 

%
%
We call $\mathcal{G}$  \textit{simple} if the complexification of its Lie algebra is simple.  
We call two simple Lie groups \textit{twins} if one is locally isomorphic to $\mathbf{SO}(n+1,n)$ and the other is locally isomorphic to $\mathbf{Sp}_{2n}(\R)$.

Given a locally symmetric orbifold $\frakX_{\Gamma}$, we define a \textit{geodesic} on $\frakX_{\Gamma}$ to be a curve that lifts to a geodesic on $\frakX$. We note that some authors use the term geodesic to instead refer to locally length minimizing curves.
Whereas these two notions coincide on manifolds, this is not the case on orbifolds \cite[2.4]{GordonSurvey}.
The \textit{rational length spectrum} of $\frakX_{\Gamma}$ is the set
$$\Q L(\frakX_{\Gamma})=\{a\lambda\ | \mbox{ where $a\in \Q$ and $\lambda$ is the length of a closed geodesic in $\frakX_{\Gamma}$}\}.$$ 
Two orbifolds $\frakX_{\Gamma_1}$ and $\frakX_{\Gamma_2}$ are \textit{length-commensurable} if $\Q L(\frakX_{\Gamma_1})=\Q L(\frakX_{\Gamma_2})$.
It is easy to see, for instance, that commensurable orbifolds are always length-commensurable.


\begin{thma}\label{thm:main}
Let $\frakX_{\Gamma_1}=\frakX_1/\Gamma_1$ and $\frakX_{\Gamma_2}=\frakX_2/\Gamma_2$ be compact locally symmetric spaces of nonpositive curvature which are associated to simple Lie groups that are not twins.
Suppose $\frakX_{\Gamma_1}$ is a manifold and $\frakX_{\Gamma_2}$ is an orbifold with nontrivial singular set, and if either of the symmetric spaces $\frakX_{\Gamma_1}$ or $\frakX_{\Gamma_2}$ have real rank greater than one, assume Schanuel's conjecture in transcendental number theory.
If $\frakX_{\Gamma_1}$ and $\frakX_{\Gamma_2}$ are length-commensurable then $\frakX_{\Gamma_1}$ and $\frakX_{\Gamma_2}$ cannot be isospectral. 
\end{thma}

Recall that the work of Duistermaat and Guillemin \cite{DuistermaatGuillemin} and Duistermaat, Kolk, and Varadarajan \cite{DKV} shows that the set of lengths of closed geodesics on a compact locally symmetric manifold of nonpositive curvature are determined by the manifold's Laplace spectrum (see also \cite[Theorem 10.1]{PrasadRap}). Were this result to be extended to compact locally symmetric orbifolds of nonpositive curvature, our \hyperref[thm:main]{Main Theorem} could be strengthened in that the assumption that $\frakX_{\Gamma_1}$ and $\frakX_{\Gamma_2}$ are length-commensurable could be removed.

The trace formula of Duistermaat and Guillemin \cite{DuistermaatGuillemin} has in fact been extended to the orbifold setting \cite{stanhope-uribe, sandoval}. However, the Maslov indices appearing in the orbifold trace formula do not have a variational interpretation as in the manifold setting, where the Maslov index of a closed geodesic is equal to its Morse index in the space of closed loops and hence is zero when the manifold in question has nonpositive curvature. In particular it is not known whether or not they are all zero in the context of this paper. This prevents one from concluding that the eigenvalues of the Laplacian determine the set of lengths of closed geodesics on a compact locally symmetric orbifold of nonpositive curvature.\\


%

\paragraph*{\textbf{Acknowledgements.}}

The authors would like to thank Mikhail Belolipetsky for initially bringing this problem to their attention and Carolyn Gordon, Alejandro Uribe and David Webb for useful conversations on the material in this article. The first author was partially supported by an NSF RTG grant DMS-1045119 and an NSF Mathematical Sciences Postdoctoral Fellowship.




\section{Proof of Main Theorem}

%

We now prove our \hyperref[thm:main]{Main Theorem}. 
We will do so by contradiction. 
To that end suppose that $\frakX_{\Gamma_1}$ and $\frakX_{\Gamma_2}$ are isospectral.
By Selberg's lemma there exists a finite index, torsion free subgroup $\Gamma_2^\prime \subset \Gamma_2$. 
In particular $\frakX_{\Gamma_2^\prime}$ is a finite degree manifold cover of $\frakX_{\Gamma_2}$. A consequence of this is that the manifolds $\frakX_{\Gamma_1}$ and $\frakX_{\Gamma_2^\prime}$ are length-commensurable. 

In \cite[Cor. 8.14]{PrasadRap}, it was shown that length-commensurability of  $\frakX_{\Gamma_1}$ and $\frakX_{\Gamma_2'}$ implies a technical algebraic condition called \textit{weak commensurability} relating the lattices $\Gamma_1\subset \mathcal{G}_1$ and $\Gamma_2'\subset \mathcal{G}_2$.   
This implication is unconditional when  $\mathrm{rank}(\frakX_{\Gamma_1})=\mathrm{rank}(\frakX_{\Gamma_2})=1$, but for higher rank is conditional upon the truth of Schanuel's conjecture.   
The precise formulation of this condition requires realizing our Lie groups as the real points of algebraic $\R$-groups and we refer the reader to  \cite[Def. 1.1]{PrasadRap} for details.  

We now show that weak commensurability of lattices determines the local geometry of their corresponding locally symmetric spaces.

\begin{theorem}\label{isometriccover}
Let $\mathcal{G}_1$ and $\mathcal{G}_2$ be nontwin, noncompact, simple, semisimple Lie groups with trivial centers
and let $\Gamma_1\subset \mathcal{G}_1$ and $\Gamma_2'\subset \mathcal{G}_2$ be weakly commensurable lattices.  Then $\mathfrak{X}_{\Gamma_1}$ and $\mathfrak{X}_{\Gamma_2'}$ have isometric universal covers.
\end{theorem}

Before we can prove this theorem, we will need a technical lemma.
Let $\mathfrak{g}$ and $\mathfrak{k}$ denote the Lie algebras of $\mathcal{G}$ and $\mathcal{K}$ respectively and let $\mathfrak{g}=\mathfrak{k}\oplus \mathfrak{a}\oplus \mathfrak{n}$ be an Iwasawa decomposition.
The \textit{$\R$-rank} of  $\mathcal{G}$, denoted $\mathrm{rank}_{\R}(\mathcal{G})$, is the dimension of $\mathfrak{a}$.
It is well known that $\mathrm{rank}(\mathfrak{X})$, the dimension of a maximal flat in $\mathfrak{X}$, is equal to $\mathrm{rank}_{\R}(\mathcal{G})$.

\begin{lemma}\label{typeandrank}
Let $\mathcal{G}_1$ and $\mathcal{G}_2$ be connected, noncompact, simple, semisimple Lie groups
and let $\mathfrak{X}_1$  and $\mathfrak{X}_2$ be their associated globally symmetric spaces.  Then $\mathfrak{X}_1$ and $\mathfrak{X}_2$ are isometric if and only if 
\begin{enumerate}[\quad (S1)]
\item $\mathrm{rank}(\mathfrak{X}_1)=\mathrm{rank}(\mathfrak{X}_2)$,
\item $\mathfrak{g}_1\otimes_\R \C$ and $\mathfrak{g}_2\otimes_\R \C$ are isomorphic as complex Lie algebras, 
\item and, in the case that $\mathfrak{g}_1\otimes_\R \C$ and $\mathfrak{g}_2\otimes_\R \C$ are either both  isomorphic to $\mathfrak{sl}_n(\C)$ or both isomorphic to $\mathfrak{so}_{2n}(\C)$,  $\mathfrak{k}_1$ and $\mathfrak{k}_2$ are isomorphic as real Lie algebras.
\end{enumerate}
\end{lemma}

\begin{proof}
If $\mathfrak{X}_1$ and $\mathfrak{X}_2$ are isometric, it follows that $\mathcal{G}_1/Z(\mathcal{G}_1)$ and $\mathcal{G}_2/Z(\mathcal{G}_2)$  are isomorphic, where $Z(\mathcal{G}_i)$ denotes the center of $\mathcal{G}_i$.  Properties (S1)-(S3) readily follow. 
Now suppose properties (S1)-(S3).  By \cite[Cor. VI.1.3 ]{Hel}, its suffices to show that $\mathfrak{g}_1$ and $\mathfrak{g}_2$ are isomorphic as real Lie algebras.  By the remarks after \cite[Thm. 6.2]{Hel}, (S1) and (S2) imply this result in all cases except when $\mathfrak{g}_1\otimes_\R \C$ and $\mathfrak{g}_2\otimes_\R \C$ are either both  isomorphic to $\mathfrak{sl}_n(\C)$ or are both isomorphic to $\mathfrak{so}_{2n}(\C)$.  In this case, (S3) together with \cite[Thm. 6.2]{Hel} imply the desired result.
\end{proof}

\begin{proof}[Proof of Theorem \ref{isometriccover}]
In light of \cite[Theorem 1]{PrasadRap} and  \cite[Theorem 1.2]{GarRap}, weak commensurability of $\Gamma_1$ and $\Gamma_2'$ implies either $\mathfrak{g}_1\otimes_\R \C$ and $\mathfrak{g}_2\otimes_\R \C$ are isomorphic as complex Lie algebras or $\mathcal{G}_1$ and $\mathcal{G}_2$ are twins.
By our assumption, we may assume the former.
Suppose that one of the lattices is arithmetic.  
By \cite[Theorem 7]{PrasadRap} both are arithmetic.  
By \cite[Theorem 6]{PrasadRap}, the Tits index of the algebraic $\R$-groups associated to these two groups are the same,  hence $\mathrm{rank}_\R(\mathcal{G}_1)=\mathrm{rank}_\R(\mathcal{G}_2)$, from which it follows $\mathrm{rank}(\mathfrak{X}_1)=\mathrm{rank}(\mathfrak{X}_2)$.  
In the two cases of (S3), the Tits index determines the isomorphism classes of $\mathfrak{k}_1$ and $\mathfrak{k}_2$ and hence by Lemma \ref{typeandrank} the result holds.  Now suppose neither lattice is arithmetic.  By the results of \cite{Mar} and \cite{GroSch}, and the observation that $\mathfrak{g}_1\otimes_\R \C$ and $\mathfrak{g}_2\otimes_\R \C$ are isomorphic, 
either $\mathcal{G}_1$ and $\mathcal{G}_2$ are both locally isomorphic as Lie groups to $\mathbf{SU}(n,1)$ or are both locally isomorphic as Lie groups to $\mathbf{SO}(n,1)$.   
In either case, the result holds.
\end{proof}

In particular this shows that the universal covers of $\frakX_{\Gamma_1}$ and $\frakX_{\Gamma_2}$ are isometric. The proof of our \hyperref[thm:main]{Main Theorem} now follows from Proposition 3.4 of \cite{GorRos}, which implies that $\frakX_{\Gamma_1}$ and $\frakX_{\Gamma_2}$ cannot have a common Riemannian cover and in particular have nonisometric universal covers.


\begin{rmk}
Theorem \ref{isometriccover} is a rigidity result, which, when restated in the language of algebraic groups, 
can be viewed as a weakening of the assumptions of Mostow--Prasad rigidity \cite[Theorem B]{Prasad} from isomorphic lattices to weakly commensurable lattices.
As this formulation may be of independent interest, we state it here.
\end{rmk}

\begin{theoremrefom}\label{rgroups}
Let $\mathbf{G}_1$ and $\mathbf{G}_2$ be connected, $\R$-isotropic, absolutely simple, semisimple, algebraic $\R$-groups whose $\R$-points are not twins.
If $\Gamma_1\subset \mathbf{G}_1(\R)$ and $\Gamma_2'\subset \mathbf{G}_2(\R)$ are weakly commensurable lattices then $\mathbf{G}_1$ and $\mathbf{G}_2$ are $\R$-isomorphic.
\end{theoremrefom}


\end{document}